\documentclass{amsart}

\usepackage{amsthm,amsmath,amssymb}
\usepackage[numbers]{natbib}
\usepackage{tikz}
\usepackage{booktabs}
\usepackage{algorithm}
\usepackage{algorithmic}
\usepackage{afterpage}
\usepackage{caption}

\usetikzlibrary{matrix}
\usetikzlibrary{positioning}

\newtheorem{theorem}{Theorem}

\newtheorem{proposition}[theorem]{Proposition}

\newcommand{\etal}{\textit{et al.}}
\newcommand{\Z}{\mathbb{Z}}

\newcommand{\zero}{\hat{0}}

\tikzset{edge0/.style=dashed}                  % not visited yet
\tikzset{edgea/.style={red, ultra thick,->}}   % addition in red
\tikzset{edgeb/.style={very thick,->}}        % addition in black
\tikzset{edger/.style={very thick,<-}}        % addition in black, reverse direction
\tikzset{edges/.style={cyan,ultra thick,->}}   % subtraction
\tikzset{edgex/.style=ultra thick}             % done

\tikzset{eld/.style={draw, circle, inner sep=0, minimum size=0.15cm}}
\tikzset{ele/.style={draw, circle, inner sep=0.05cm, minimum size=0.4cm}}
\tikzset{ell/.style={draw, circle, inner sep=0.05cm, minimum size=0.5cm}}
\tikzset{elh/.style={draw, circle, inner sep=0.05cm, minimum size=0.5cm, blue, ultra thick}}

\title{Fast M\"obius inversion in semimodular lattices and U-labelable posets}

\author{Petteri Kaski}
\author{Jukka Kohonen}
\address{Department of Computer Science\\ Aalto University, Espoo, Finland}
\email{petteri.kaski@aalto.fi}
\email[Corresponding author]{jukka.kohonen@aalto.fi}

\author{Thomas Westerb\"ack}
\address{Department of Mathematics and Systems Analysis\\ Aalto University, Espoo, Finland}
\email{thomas.westerback@aalto.fi}

\begin{document}

\begin{abstract}
  We consider the problem of fast zeta and M\"obius transforms in
  finite posets, particularly in lattices.  It has previously been
  shown that for a certain family of lattices, zeta and M\"obius
  transforms can be computed in $O(e)$ elementary arithmetic
  operations, where $e$ denotes the size of the covering relation.  We
  show that this family is exactly that of geometric lattices.  We
  also extend the algorithms so that they work in $e$ operations for
  all semimodular lattices, including chains and divisor lattices.
  Finally, for both transforms, we provide a more general algorithm
  that works in $e$~operations for all R-labelable posets and their
  non-graded generalization, which we call U-labelable.

  \vspace{0.2cm}
\end{abstract}

% 06C10 Semimodular lattices
% 06A07 Combinatorics of partially ordered sets
% 68W40 Computer science: Analysis of algorithms
% 68R05 Computer science: Combinatorics

\maketitle

%%%%%%%%%%%%%%%%%%%%%%%%%%%%%%%%%%%%%%%%%%%%%%%%%%%%%%%
\section{Introduction}
\label{sec:intro}

Fast methods for computing zeta and M\"obius transforms in finite
posets are important for many algorithms of combinatorial nature, such
as graph coloring \cite{bjorklund2007} and fast Fourier transforms on
inverse semigroups \cite{malandro2010}.  From an algebraic
perspective, these transforms are basis-changing isomorphisms
analogous to the fast Fourier transform \cite{bjorklund2015}.  From a
computational perspective, the zeta transform is related to
\emph{ensemble computation}, where one is required to compute several
different products of inputs, possibly sharing common subexpressions,
and the challenge is to minimize the number of elementary
multiplications perfomed (see Garey and Johnson~\cite{garey1979}).

Let $(P,\le)$ be a finite partially ordered set, or \emph{poset}, with
$|P|=v$ elements, and let $f : P \to A$ be a function to an abelian
group $A$.  The \emph{zeta transform} of $f$ is the function $g:
P \to A$ such that, for all $y \in P$,
\begin{equation*}
  g(y) = \sum_{x \le y} f(x).
\end{equation*}

For computational purposes, the function $f$ can be represented as a
$v$-element row vector $\vec{f} = (f(x))_{x \in P}$, and $g$
similarly.  The transformation $\vec{f} \mapsto \vec{g}$ is linear,
and defined by the $v \times v$ matrix~$\zeta$, where $\zeta_{xy} = 1$
if $x\le y$, and $0$ otherwise.  The obvious way to compute the zeta
transform is to perform the matrix-vector multiplication
$\vec{g}=\vec{f}\zeta$, incurring $O(v^2)$ elementary additions.
However, the special structure of the transformation gives hope of
performing it faster, in particular for posets of some specific form.

For a given poset, we represent the computation from~$f$ to its zeta
transform~$g$ as a \emph{straight-line program} (see B\"urgisser
\etal~\cite{burgisser1997}), a sequence of elementary pairwise
arithmetic operations (additions and subtractions) to be executed in
turn, without loops or conditional statements.  We seek to minimize
the length of the program, that is, the number of arithmetic
operations performed.

The zeta transform is always invertible, since $\zeta$ is an upper
triangular matrix with ones on the main diagonal, assuming the matrix
lists the elements of $P$ in a linear order that is an extension
of~$P$.  We also want to construct a short straight-line program for
the inverse transform $\mu = \zeta^{-1}$, called the \emph{M\"obius
  transform}.  The matrices $\zeta$ and $\mu$ are in fact
representations of the zeta and M\"obius functions of the poset.  We
refer to Stanley~\cite{stanley2011} for further discussion.

%%%%%%%%%%%%%%%%%%%%%%%%%%%%%%%%%%%%%%%%%%%%%%%%%%%%%%%
\subsection{Measures of the complexity}

We relate the program length to certain parameters characterizing the
complexity of the poset: the number of elements or vertices ($v$), the
number of join-irreducible elements ($n$), and the number of edges in
the Hasse diagram of the poset ($e$).

In a lattice we always have $v-1 \le e \le vn$.\footnote{The first
  inequality holds because every non-maximal element has at least one
  upward edge.  The latter inequality follows by considering the
  upward edges of an arbitrary element $x$ and the spectrum map
  $\varphi$, as defined in Section~\ref{sec:preliminaries}.  If both
  $y \gtrdot x$ and $z \gtrdot x$, then $\varphi(y) \setminus
  \varphi(x)$ and $\varphi(z) \setminus \varphi(x)$ must be disjoint,
  otherwise $y$ and $z$ would have two incomparable lower bounds.
  Hence $x$ has at most $n$ upward edges, and the whole lattice has $e
  \le vn$.}  As~an example where $e$ is large, consider the lattice of
subsets of an $n$-element set.  This lattice has $v=2^n$ elements, of
which $n$ are join-irreducible (the singletons), and $e=2^{n-1}n =
\Theta(vn)$ edges.  Both zeta and M\"obius transforms can be computed
in $\Theta(e)=\Theta(vn)$ operations by Yates's
construction~\cite{yates1937}.  As~an example where $e$ is small,
consider the $v$-element chain, which has $n=e=v-1$, thus $vn =
\Theta(e^2)$.  In a chain both transforms are easily computed in $e$
operations.

The parameter $e$ also appears in a lower bound by
Kennes~\cite{kennes1992}: for any lattice, any sequence of additions
that computes the zeta transform has length at least~$e$.  Note that
this lower bound does not apply to posets in general; as a
counterexample, Bj\"orklund \etal~\cite[\S 6]{bjorklund2015} exhibit a
non-lattice bipartite poset whose zeta transform can be computed in
$O(\sqrt{e})$ additions, using a construction by
Valiant~\cite{valiant1986}.

For any lattice, Bj\"orklund \etal~\cite{bjorklund2015} construct both
programs with length $O(vn)$.  This upper bound is always valid but
possibly quite crude.  A tighter length bound $O(e)$ applies for
lattices fulfilling the condition that for any element~$x$ and any
join-irreducible element~$i$ that is not below $x$, the join $x \vee
i$ covers~$x$.  Examples of such lattices include the subset lattice,
the partition lattice, and the lattice of subspaces of a finite vector
space.  For some lattices, the resulting program length is highly
dependent upon the choice of certain details in the construction.  As
a simple example, with the $v$-element chain one choice leads to the
optimal $O(e)$ additions, while another choice leads to $\Theta(e^2)$
additions.  We seek here a better understanding of posets that admit
fast computation of the zeta and M\"obius transforms.

%%%%%%%%%%%%%%%%%%%%%%%%%%%%%%%%%%%%%%%%%%%%%%%%%%%%%%%
\subsection{Our results}

First we show that the aforementioned condition by Bj\"orklund
\etal~\cite{bjorklund2015} is equivalent to the the lattice being
\emph{geometric} (semimodular and atomic).  Thus in any geometric
lattice both transforms can be done in $O(e)$ arithmetic operations.
This result also provides an alternative ``single-axiom''
characterization of geometric lattices.

Second, we show that atomicity is not needed: in any
\emph{semimodular} lattice, the construction can be done in a manner
that yields a program of length~$e$, matching Kennes's lower bound.
The optimal straight-line programs for chains now follow as a special
case, as well as for other non-atomic semimodular lattices such as the
lattice of positive divisors of an integer.

Third, for further generality, we show that if a poset admits a
certain kind of \emph{edge labeling}, then the transforms can be done
in exactly $e$ arithmetic operations by following the edges in an
order implied by their labels.  The idea of such labelings goes back
to Stanley~\cite{stanley1972}, who studied edge labeling in the
context of counting chains in a lattice; for bibliography see
Stanley~\cite[Chapter~3]{stanley2011}.  The class of posets admitting
zeta and M\"obius transforms in $e$~operations now encompasses the
family of \emph{R-labelable} posets, and their non-graded
generalization, which we will call \emph{U-labelable}.  Semimodular
lattices are included as a special case, as well as lower semimodular
lattices (by duality) and supersolvable lattices.

The new results are summarized in Table~\ref{table:summary} in context
with previous results.

\begin{table}[H]
  \centering
  \begin{tabular}{lll}
    \toprule
    Poset family & Program length & Reference \\
    \midrule
    all lattices            & $O(v^2)$     & matrix-vector multiplication \\
    subset lattices         & $O(vn)=O(e)$ & Yates~\cite{yates1937} \\
    distributive lattices (zeta only)   & $O(vn)$      & Parviainen and Koivisto \cite{parviainen2010} \\
    all lattices            & $O(vn)$      & Bj\"orklund \etal~\cite{bjorklund2015} \\
    lattices with a certain condition & $O(e)$       & Bj\"orklund \etal~\cite{bjorklund2015} \\
    geometric lattices      & $O(e)$       & new (\S2) \\
    semimodular lattices    & $O(e)$       & new (\S3) \\
    R-, U-labelable posets  & $O(e)$       & new (\S4) \\
    \bottomrule
  \end{tabular}
  \caption{Program lengths for the zeta and M\"obius
    transforms.}
  \label{table:summary}
\end{table}

%%%%%%%%%%%%%%%%%%%%%%%%%%%%%%%%%%%%%%%%%%%%%%%%%%%%%%%
\section{Preliminaries}
\label{sec:preliminaries}

For a general treatment on posets and lattices we refer to
Gr\"atzer~\cite{gratzer2011} and Stanley~\cite{stanley2011}.  In the
present work a poset $P$ or a lattice $L$ is always finite.  The
covering relation is denoted by $\lessdot$, and the set of covering
pairs, or edges of the Hasse diagram, is $E = \{(x,y) : x \lessdot
y\}$.  An element $x \in P$ is \emph{join-irreducible} if it covers
exactly one element, and $I$ is the set of join-irreducible elements.
Poset size is characterized by three quantities: the number of
elements $v=|P|$; the number of join-irreducible elements $n=|I|$; and
the number of edges $e = |E|$.

% Note that Grätzer uses "atomistic" for atomic, and "atomic" for
% something else.

% This definition of semimodularity is found in Grätzer's Exercise 2.1.

Without loss of generality, we assume that the join-irreducible
elements of a lattice are denoted by integers $I=[n]$.  The
\emph{spectrum map} of an element $x \in L$ is $\varphi(x) = \{h \in I
: h \le x\}$, and for $i \le n$, the \emph{prefix spectrum map} is
$\varphi_i(x) = \{h \in [i] : h \le x\}$, with $\varphi_0(x) =
\varnothing$.

In a lattice the minimum element is $\zero$, and an element that
covers $\zero$ is an \emph{atom}.  A~lattice is \emph{atomic} if every
element is a join of atoms.  A~lattice is \emph{(upper) semimodular},
if for any two elements $x$ and $y$ such that $x$ covers $x \wedge y$,
the join $x \vee y$ in turn covers $y$.  The dual of a semimodular
lattice is \emph{lower semimodular}.  An~atomic semimodular lattice is
\emph{geometric}.  Examples of geometric lattices include subset
lattices and partition lattices; examples of non-atomic semimodular
lattices include chains and divisor lattices (see Stanley \citep[\S
  3.3]{stanley2011}).

Theorem 1.3 of Bj\"orklund \etal~\cite{bjorklund2015} shows that both
zeta and M\"obius transforms can be done in $O(e)$ arithmetic
operations in any finite lattice $L$ fulfilling the following
condition:
\begin{equation}
  x \vee i \gtrdot x \qquad \text{for all $x \in L$, $i \in I$, $x \not\ge i$}.
  \label{eq:x}
\end{equation}
The following theorem provides an alternative characterization.

\begin{theorem} 
  A finite lattice fulfills condition \eqref{eq:x} if and only if it
  is geometric.
  \label{thm:geom}
\end{theorem}
\begin{proof}
  For the ``if'' direction, let $L$ be a finite geometric lattice,
  $x\in L$, and $i\in I$ such that $x \not\ge i$.  Since $L$ is atomic,
  $i$ is an atom, and $i \gtrdot \zero = x \wedge i$.  Now by
  semimodularity $x \vee i \gtrdot x$.

  For the ``only if'' direction, let $L$ be a finite lattice where
  \eqref{eq:x} holds.  For any join-irreducible $i$, choosing
  $x=\zero$ shows that $i \gtrdot \zero$, thus $L$ is atomic.

  Let then $x$ and $y$ be elements such that $y \gtrdot x \wedge y = m$.
  Choose a join-irreducible element $i$ such that $i \le y$ but $i
  \not\le x$.  Now $y=m \vee i$, so $x \vee y = x \vee m \vee i = x
  \vee i$, which covers $x$ by condition \eqref{eq:x}.  Thus $L$ is
  semimodular.
\end{proof}

%%%%%%%%%%%%%%%%%%%%%%%%%%%%%%%%%%%%%%%%%%%%%%%%%%%%%%%
\section{Fast M\"obius inversion in semimodular lattices}

We shall next show that atomicity is unnecessary for fast zeta and
M\"obius transforms: semimodularity by itself suffices to ensure
programs of length~$e$.

%%%%%%%%%%%%%%%%%%%%%%%%%%%%%%%%%%%%%%%%%%%%%%%%%%%%%%%
\subsection{Fast zeta transform}

We recall the fast zeta transform algorithm of Bj\"ork\-lund
\etal~\cite[\S 3.2]{bjorklund2015}.  While the original algorithm is
described by embedding the lattice $L$ into a set family
$\mathcal{L}$, for the sake of transparency we describe the algorithm
directly in terms of lattice operations.

\begin{algorithm}[b]
  \caption{Fast zeta transform}
  \begin{algorithmic}[1]
    \REQUIRE Lattice $L$ with join-irreducibles $I=[n]$
    \ENSURE Straight-line program that, given $f$, computes its zeta transform $g$
    \FORALL{$x \in L$}
    \PRINT ``$g(x) \leftarrow f(x)$'' \COMMENT{Initialization}
    \ENDFOR
    \FOR{$i = 1, 2, \ldots, n$}
    \FORALL{$x \in L$ such that $x \not\ge i$}
    \STATE $y \leftarrow x \vee i$
    \IF {$\varphi_{i-1}(x) = \varphi_{i-1}(y)$}
    \PRINT ``$g(y) \leftarrow g(y) + g(x)$'' \COMMENT{Addition}
    \ENDIF
    \ENDFOR
    \ENDFOR
  \end{algorithmic}
  \label{alg:fastzeta}
\end{algorithm}

Algorithm~\ref{alg:fastzeta} takes as input a representation of the
lattice structure, and outputs a sequence of additions to be
performed.  Theorem~1.1 of Bj\"orklund \etal~\cite{bjorklund2015}
shows that the resulting straight-line program indeed computes the
zeta transform for any lattice, regardless of how $I$~is ordered.  The
outer loop beginning on line~4 is executed $n$ times, and the inner
loop at most $|L|=v$ times, so at most $O(vn)$ additions are used; but
the precise number may be much smaller, depending on lattice structure
and the ordering of~$I$.  Consider a chain of $v$ elements, with
$n=e=v-1$.  Bottom-up traversal incurs $e$~additions, which is the
optimal number; each addition proceeds along an edge of the chain.  In
contrast, top-down traversal is suboptimal and incurs $\Theta(e^2)$
additions \cite{bjorklund2015}, many of which involve two distant
elements that do not share an edge.

The observation that bottom-up traversal leads to fewer additions can
be generalized to all semimodular lattices.  The following theorem
provides the optimal arrangement, whereby each addition corresponds to
an edge of the lattice.

\begin{theorem} 
  In a semimodular lattice with $e$ edges, the join-irreducible
  elements can be ordered so that Algorithm~\ref{alg:fastzeta}
  generates exactly $e$ additions.
  \label{thm:semizeta}
\end{theorem}
\begin{proof}
  Order the join-irreducible elements by increasing rank, breaking
  ties arbitrarily.  Consider the situation when the condition on
  line~7 succeeds.

  Since $i$ is join-irreducible, there is a unique element $k$ such
  that $k \lessdot i$.  Since $k$ has a rank strictly smaller than
  $i$, it follows that $\varphi_{i-1}(k) = \varphi(k)$.  Now, because
  $k \le i \le x \vee i = y$, we have
  \[
  \varphi(k) = \varphi_{i-1}(k) \subseteq \varphi_{i-1}(y) = \varphi_{i-1}(x)
  \subseteq \varphi(x),
  \]
  where the equality $\varphi_{i-1}(y) = \varphi_{i-1}(x)$ is true due
  to the condition on line~7.
  
  Since $\varphi(k) \subseteq \varphi(x)$, it follows that $k \le x$.
  Thus since $i \not\le x$ we must have $x \wedge i = k$, implying $i
  \gtrdot k = x \wedge i$, and then $x \vee i \gtrdot x$ by
  semimodularity.

  We have seen that whenever an addition is generated on line~8, it
  involves elements $x$ and $y$ such that $x \lessdot y$, so we can
  associate the addition with an edge of the Hasse diagram.  No two
  additions are associated with the same edge, so the number of
  additions is at most~$e$.  Conversely, each edge is visited by the
  algorithm, since if $x \lessdot y$, then $\varphi(x) \ne
  \varphi(y)$, and the edge is visited with the smallest $i$ such that
  $\varphi_i(x) \ne \varphi_i(y)$.  Hence the algorithm perfoms
  exactly $e$ additions.
\end{proof}

Figure~\ref{fig:fastzeta} illustrates how the zeta transform proceeds
in a semimodular lattice that has 9~edges.  Nine additions are
performed in four phases, corresponding to the four join-irreducible
elements of the lattice.

\begin{figure}[p]
  \begin{tikzpicture}
    \matrix(a)[matrix of math nodes, column sep=0cm, row sep=0.5cm]{
      && \node[ell](6){6}; \\
      \node[ell](3){3}; && \node[ell](5){5}; && \node[ell](4){4}; \\
      & \node[elh](1){1}; && \node[ell](2){2}; \\
      && \node[ell](0){0}; \\};
    \draw[edgea] (0)--(1);
    \draw[edgea] (2)--(5);
    \draw[edgea] (4)--(6);
    \draw[edge0] (0)--(2);
    \draw[edge0] (1)--(5);
    \draw[edge0] (3)--(6);
    \draw[edge0] (1)--(3);
    \draw[edge0] (5)--(6);
    \draw[edge0] (2)--(4);
  \end{tikzpicture}
  \hfill
  \begin{tikzpicture}
    \matrix(a)[matrix of math nodes, column sep=0cm, row sep=0.5cm]{
      && \node[ell](6){6}; \\
      \node[ell](3){3}; && \node[ell](5){5}; && \node[ell](4){4}; \\
      & \node[ell](1){1}; && \node[elh](2){2}; \\
      && \node[ell](0){0}; \\};
    \draw[edgex] (0)--(1);
    \draw[edgex] (2)--(5);
    \draw[edgex] (4)--(6);
    \draw[edgea] (0)--(2);
    \draw[edgea] (1)--(5);
    \draw[edgea] (3)--(6);
    \draw[edge0] (1)--(3);
    \draw[edge0] (5)--(6);
    \draw[edge0] (2)--(4);
  \end{tikzpicture}
  \hfill
  \begin{tikzpicture}
    \matrix(a)[matrix of math nodes, column sep=0cm, row sep=0.5cm]{
      && \node[ell](6){6}; \\
      \node[elh](3){3}; && \node[ell](5){5}; && \node[ell](4){4}; \\
      & \node[ell](1){1}; && \node[ell](2){2}; \\
      && \node[ell](0){0}; \\};
    \draw[edgex] (0)--(1);
    \draw[edgex] (2)--(5);
    \draw[edgex] (4)--(6);
    \draw[edgex] (0)--(2);
    \draw[edgex] (1)--(5);
    \draw[edgex] (3)--(6);
    \draw[edgea] (1)--(3);
    \draw[edgea] (5)--(6);
    \draw[edge0] (2)--(4);
  \end{tikzpicture}
  \hfill
  \begin{tikzpicture}
    \matrix(a)[matrix of math nodes, column sep=0cm, row sep=0.5cm]{
      && \node[ell](6){6}; \\
      \node[ell](3){3}; && \node[ell](5){5}; && \node[elh](4){4}; \\
      & \node[ell](1){1}; && \node[ell](2){2}; \\
      && \node[ell](0){0}; \\};
    \draw[edgex] (0)--(1);
    \draw[edgex] (2)--(5);
    \draw[edgex] (4)--(6);
    \draw[edgex] (0)--(2);
    \draw[edgex] (1)--(5);
    \draw[edgex] (3)--(6);
    \draw[edgex] (1)--(3);
    \draw[edgex] (5)--(6);
    \draw[edgea] (2)--(4);
  \end{tikzpicture}
  \\[0.25cm]
  \begin{minipage}{0.55\textwidth}
    \captionsetup{width=\linewidth}
    \caption{Fast zeta transform in a semimodular lattice.  In each
      phase the join-irreducible element $i$ being considered is
      highlighted in blue.  Red arrows indicate addition along edges.
      Thick black edges have been visited already, and dashed edges
      are yet unvisited.}
    \label{fig:fastzeta}
  \end{minipage}
  \hfill
  \begin{tabular}{ll}
    \toprule
    \multicolumn{2}{l}{Straight-line program} \\
    \midrule
    Phase 1 & $g(1) \leftarrow g(1) + g(0)$ \\
    & $g(5) \leftarrow g(5) + g(2)$ \\
    & $g(6) \leftarrow g(6) + g(4)$ \\
    \midrule
    Phase 2 & $g(2) \leftarrow g(2) + g(0)$ \\
    & $g(5) \leftarrow g(5) + g(1)$ \\
    & $g(6) \leftarrow g(6) + g(3)$ \\
    \midrule
    Phase 3 & $g(3) \leftarrow g(3) + g(1)$ \\
    & $g(6) \leftarrow g(6) + g(5)$ \\
    \midrule
    Phase 4 & $g(4) \leftarrow g(4) + g(2)$ \\
    \bottomrule
  \end{tabular}
  \\[1cm]
  \begin{tikzpicture}
    \matrix(a)[matrix of math nodes, column sep=0cm, row sep=0.5cm]{
      && \node[ell](6){6}; \\
      \node[ell](3){3}; && \node[ell](5){5}; && \node[elh](4){4}; \\
      & \node[ell](1){1}; && \node[ell](2){2}; \\
      && \node[ell](0){0}; \\};
    \draw[edge0] (0)--(1);
    \draw[edge0] (2)--(5);
    \draw[edge0] (4)--(6);
    \draw[edge0] (0)--(2);
    \draw[edge0] (1)--(5);
    \draw[edge0] (3)--(6);
    \draw[edge0] (1)--(3);
    \draw[edge0] (5)--(6);
    \draw[edges] (2)--(4);
  \end{tikzpicture}
  \hfill
  \begin{tikzpicture}
    \matrix(a)[matrix of math nodes, column sep=0cm, row sep=0.5cm]{
      && \node[ell](6){6}; \\
      \node[elh](3){3}; && \node[ell](5){5}; && \node[ell](4){4}; \\
      & \node[ell](1){1}; && \node[ell](2){2}; \\
      && \node[ell](0){0}; \\};
    \draw[edge0] (0)--(1);
    \draw[edge0] (2)--(5);
    \draw[edge0] (4)--(6);
    \draw[edge0] (0)--(2);
    \draw[edge0] (1)--(5);
    \draw[edge0] (3)--(6);
    \draw[edges] (1)--(3);
    \draw[edges] (5)--(6);
    \draw[edgex] (2)--(4);
  \end{tikzpicture}
  \hfill
  \begin{tikzpicture}
    \matrix(a)[matrix of math nodes, column sep=0cm, row sep=0.5cm]{
      && \node[ell](6){6}; \\
      \node[ell](3){3}; && \node[ell](5){5}; && \node[ell](4){4}; \\
      & \node[ell](1){1}; && \node[elh](2){2}; \\
      && \node[ell](0){0}; \\};
    \draw[edge0] (0)--(1);
    \draw[edge0] (2)--(5);
    \draw[edge0] (4)--(6);
    \draw[edges] (0)--(2);
    \draw[edges] (1)--(5);
    \draw[edges] (3)--(6);
    \draw[edgex] (1)--(3);
    \draw[edgex] (5)--(6);
    \draw[edgex] (2)--(4);
  \end{tikzpicture}
  \hfill
  \begin{tikzpicture}
    \matrix(a)[matrix of math nodes, column sep=0cm, row sep=0.5cm]{
      && \node[ell](6){6}; \\
      \node[ell](3){3}; && \node[ell](5){5}; && \node[ell](4){4}; \\
      & \node[elh](1){1}; && \node[ell](2){2}; \\
      && \node[ell](0){0}; \\};
    \draw[edges] (0)--(1);
    \draw[edges] (2)--(5);
    \draw[edges] (4)--(6);
    \draw[edgex] (0)--(2);
    \draw[edgex] (1)--(5);
    \draw[edgex] (3)--(6);
    \draw[edgex] (1)--(3);
    \draw[edgex] (5)--(6);
    \draw[edgex] (2)--(4);
  \end{tikzpicture}
  \\[0.5cm]
  \begin{minipage}{0.55\textwidth}
    \captionsetup{width=\linewidth}
    \caption{Fast M\"obius transform in a semimodular lattice.  In
      each phase the join-irreducible element $i$ being considered is
      highlighted in blue.  Cyan arrows indicate subtraction along
      edges.  Thick black edges have been visited already, and dashed
      edges are yet unvisited.}
    \label{fig:fastmobius}
  \end{minipage}
  \hfill
  \begin{tabular}{ll}
    \toprule
    \multicolumn{2}{l}{Straight-line program} \\
    \midrule
    Phase 4 & $f(4) \leftarrow f(4) - f(2)$ \\
    \midrule
    Phase 3 & $f(6) \leftarrow f(6) - f(5)$ \\
    & $f(3) \leftarrow f(3) - f(1)$ \\
    \midrule
    Phase 2 & $f(6) \leftarrow f(6) - f(3)$ \\
    & $f(5) \leftarrow f(5) - f(1)$ \\
    & $f(2) \leftarrow f(2) - f(0)$ \\
    \midrule
    Phase 1 & $f(6) \leftarrow f(6) - f(4)$ \\
    & $f(5) \leftarrow f(5) - f(2)$ \\
    & $f(1) \leftarrow f(1) - f(0)$ \\
    \bottomrule
  \end{tabular}
\end{figure}

%\afterpage{\clearpage}

%%%%%%%%%%%%%%%%%%%%%%%%%%%%%%%%%%%%%%%%%%%%%%%%%%%%%%%
\subsection{Fast M\"obius transform}

Next we shall show the corresponding result for the M\"obius
transform, which computes $f$ from~$g$.  Again we start with the
algorithm by Bj\"orklund \etal~\cite[\S 3.6]{bjorklund2015} expressed
in terms of lattice operations.

\begin{algorithm}[b]
  \caption{Fast M\"obius transform}
  \begin{algorithmic}[1]
    \REQUIRE Lattice $L$ with join-irreducibles $I=[n]$
    \ENSURE Straight-line program that, given $g$, computes its M\"obius transform $f$
    \FORALL{$x \in L$}
    \PRINT ``$f(x) \leftarrow g(x)$'' \COMMENT{Initialization}
    \ENDFOR
    \FOR{$i = n, n-1, \ldots, 1$}
    \FORALL{$x \in L$ such that $x \not\ge i$}
    \STATE $y \leftarrow x \vee i$
    \IF {$\varphi_{i-1}(x) = \varphi_{i-1}(y)$}
    \PRINT ``$f(y) \leftarrow f(y) - f(x)$'' \COMMENT{Subtraction}
    \ENDIF
    \ENDFOR
    \ENDFOR
  \end{algorithmic}
  \label{alg:fastmobius}
\end{algorithm}

Algorithm~\ref{alg:fastmobius} is similar to the zeta transform
algorithm, with a few crucial changes.  The join-irreducible elements
are visited in reverse order, the roles of $f$~and~$g$ are inversed,
and subtraction replaces the addition on line~8.  Since the conditions
on lines 5--7 are the same as before, the next theorem follows.

\begin{theorem} 
  In a semimodular lattice with $e$ edges, the join-irreducible
  elements can be ordered so that Algorithm~\ref{alg:fastmobius}
  generates exactly $e$ subtractions.
  \label{thm:semimobius}
\end{theorem}
\begin{proof}
  Order and name the join-irreducible elements as $1,2,\ldots,n$ by
  increasing rank, breaking ties arbitrarily.  When the condition on
  line~7 succeeds, it holds that $x \lessdot y$, by the same reasoning
  as in Theorem~\ref{thm:semizeta}.  Thus the subtraction on line~8
  can be associated with that edge, and total the number of
  subtractions performed equals the number of edges.
\end{proof}

Figure~\ref{fig:fastmobius} illustrates how the M\"obius transform
proceeds on a semimodular lattice.  Note that join-irreducible
elements are considered in the reverse order $4,3,2,1$.

%%%%%%%%%%%%%%%%%%%%%%%%%%%%%%%%%%%%%%%%%%%%%%%%%%%%%%%
\section{Fast M\"obius inversion in posets labelable with unique rising chains}
\label{sec:ulabel}

Our next addition algorithm is constructed to proceed along the edges
of a poset.  By this we mean that the straight-line program consists
of $e$~lines of the form
\begin{equation*}
  g(y) \leftarrow g(y) + g(x),
\end{equation*}
one for each edge $(x,y)$, where $x \lessdot y$.  The edges are
visited in an order specified by an \emph{edge labeling} $\lambda: E
\to \Z$.  The algorithm requires the labels to be distinct so that the
order is unambiguous.  The inverse transform is obtained by reversing
the order of operations and replacing additions with subtractions.

\begin{algorithm}[H]
  \caption{Add along edges}
  \begin{algorithmic}[1]
    \REQUIRE Poset $P$ with injective edge labeling $\lambda: E \to \Z$
    \ENSURE Straight-line program
    \FORALL{$x \in P$}
    \PRINT ``$g(x) \leftarrow f(x)$'' \COMMENT{Initialization}
    \ENDFOR
    \FORALL{$(x,y) \in E$ in increasing order of $\lambda$}
    \PRINT ``$g(y) \leftarrow g(y) + g(x)$''  \COMMENT{Addition}
    \ENDFOR
  \end{algorithmic}
  \label{alg:addalong}
\end{algorithm}

\begin{algorithm}[H]
  \caption{Subtract along edges}
  \begin{algorithmic}[1]
    \REQUIRE Poset $P$ with injective edge labeling $\lambda: E \to \Z$
    \ENSURE Straight-line program
    \FORALL{$x \in P$}
    \PRINT ``$f(x) \leftarrow g(x)$'' \COMMENT{Initialization}
    \ENDFOR
    \FORALL{$(x,y) \in E$ in decreasing order of $\lambda$}
    \PRINT ``$f(y) \leftarrow f(y) - f(x)$''  \COMMENT{Subtraction}
    \ENDFOR
  \end{algorithmic}
  \label{alg:subalong}
\end{algorithm}

Both Algorithms~\ref{alg:addalong} and~\ref{alg:subalong} generate
exactly $e$ operations by design, in contrast to
Algorithms~\ref{alg:fastzeta} and~\ref{alg:fastmobius}, which do so
for semimodular lattices but not in general.  We will next formulate a
sufficient condition to ensure that Algorithm~\ref{alg:addalong}
computes the zeta transform (and that consequently
Algorithm~\ref{alg:subalong} computes the M\"obius transform).

Consider two comparable elements $x \le y$ and a unrefinable chain $C$
from $x$ to $y$,
\begin{equation*}
  x = x_0 \lessdot x_1 \lessdot \cdots \lessdot x_m = y.
\end{equation*}
With a labeling $\lambda: E \to \Z$, we say that $C$ is \emph{rising}
if its labels are increasing,
\begin{equation}
  \lambda(x_0,x_1) \le \lambda(x_1,x_2) \le \cdots \le \lambda(x_{m-1},x_m).
  \label{eq:rising}
\end{equation}
We allow here $\lambda$ to be non-injective for compatibility with
Stanley's definition~\cite[\S 3.14]{stanley2011}.  An edge labeling
$\lambda: E \to \Z$ is a \emph{U-labeling} (labeling with unique
rising chains) if, for each pair of elements $x \le y$, there is
exactly one rising chain from $x$ to~$y$.

Now if $\lambda$ is injective and $C$ is a rising chain from $x$ to
$y$, then the inequalities in~\eqref{eq:rising} are strict, and
Algorithm~\ref{alg:addalong} performs additions along $C$ in the chain
order.  Thus the input term $f(x)$ propagates along $C$ to the output
$g(y)$.  Conversely, if $C$ is not rising, then some of the additions
in the chain are performed out of the chain order, and $f(x)$ does
\emph{not} propagate to $g(y)$ along $C$.  Hence we have following
sufficient condition.

\begin{proposition}
  If $\lambda$ is an injective U-labeling, then the straight-line
  programs from Algorithms~\ref{alg:addalong} and~\ref{alg:subalong}
  compute the zeta and M\"obius transforms, respectively.
  \label{prop:uzeta}
\end{proposition}
\begin{proof}
  For each pair $x \le y$, the straight-line program from
  Algorithm~\ref{alg:addalong} propagates $f(x)$ up to $g(y)$ exactly
  once, along the unique rising chain from $x$ to $y$.  Thus for each
  element $y \in P$, the output $g(y)$ equals $\sum_{x \le y} f(x)$ as
  required.  Hence the result is the zeta transform.

  The program from Algorithm~\ref{alg:subalong} consists of subtractions
  that undo the additions in reverse order, hence it performs the
  M\"obius transform.
\end{proof}

Below we show two different edge labelings on a poset.  The labeling
on the left has two rising chains from $a$ to $d$, so if additions are
performed in this order, $g(d)$ will incorrectly contain the term
$f(a)$ twice.  On the right is a U-labeling, which leads to the
correct zeta transform $g(a)=f(a)$, $g(b)=f(a)+f(b)$,
$g(c)=f(a)+f(c)$, and $g(d)=f(a)+f(b)+f(c)+f(d)$.

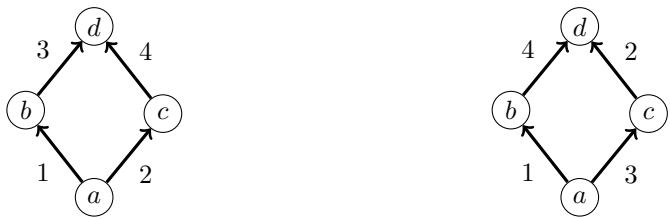
\begin{figure}[H]
  \begin{minipage}{0.49\textwidth}
    \centering
    \begin{tikzpicture}
      \matrix(a)[matrix of math nodes, column sep=0.4cm, row sep=0.6cm]{
        & \node[ell](3){d}; \\
        \node[ell](1){b}; && \node[ell](2){c}; \\
        & \node[ell](0){a}; \\};
      \draw[edgeb] (0)--(1) node[midway, below left]  {1};
      \draw[edgeb] (0)--(2) node[midway, below right] {2};
      \draw[edgeb] (1)--(3) node[midway, above left]  {3};
      \draw[edgeb] (2)--(3) node[midway, above right] {4};
    \end{tikzpicture}
  \end{minipage}
  \hfill
  \begin{minipage}{0.49\textwidth}
    \centering
    \begin{tikzpicture}
      \matrix(a)[matrix of math nodes, column sep=0.4cm, row sep=0.6cm]{
        & \node[ell](3){d}; \\
        \node[ell](1){b}; && \node[ell](2){c}; \\
        & \node[ell](0){a}; \\};
      \draw[edgeb] (0)--(1) node[midway, below left]  {1};
      \draw[edgeb] (0)--(2) node[midway, below right] {3};
      \draw[edgeb] (1)--(3) node[midway, above left]  {4};
      \draw[edgeb] (2)--(3) node[midway, above right] {2};
    \end{tikzpicture}
  \end{minipage}
\end{figure}

A poset that admits a U-labeling is \emph{U-labelable}.  A U-labeling
is not necessarily injective as required by
Algorithms~\ref{alg:addalong} and~\ref{alg:subalong}, but an injective
U-labeling can be produced as follows.

\begin{proposition}
  If a poset has a U-labeling, then it also has an injective
  U-labeling.
  \label{prop:injective}
\end{proposition}
\begin{proof}
  Let $P$ be a poset with a U-labeling $\lambda$.  Order the edges by
  $\lambda$; among edges that have the same label, order by their
  first elements according to the poset order~$\le$.  Break any
  remaining ties arbitrarily.  With the edges thus ordered, define an
  injective labeling $\lambda'$ by assigning labels $1,2,\ldots,e$ in
  order.

  Consider now an unrefinable chain $C$.  If $C$ is not rising under
  $\lambda$, then it contains three elements $s \lessdot t \lessdot u$
  such that $\lambda(s,t) > \lambda(t,u)$.  Then also $\lambda'(s,t) >
  \lambda'(t,u)$, and $C$ is not rising under $\lambda'$.

  Conversely, if $C$ is rising under $\lambda$, then with any three
  consecutive elements $s \lessdot t \lessdot u$ in $C$, we have
  either $\lambda(s,t) < \lambda(t,u)$ or $\lambda(s,t)=\lambda(t,u)$.
  In either case, by construction, $\lambda'(s,t) < \lambda'(t,u)$.
  Hence $C$ is rising under $\lambda'$.

  Since $\lambda$ and $\lambda'$ have the same rising chains,
  $\lambda'$ is also a U-labeling.
\end{proof}

By combining Propositions~\ref{prop:uzeta} and~\ref{prop:injective} we
obtain our main result for U-labelable posets.

\begin{theorem}
  In a U-labelable poset that has $e$~edges, the zeta transform can be
  computed in $e$~additions, and the M\"obius transform in
  $e$~subtractions.
  \label{thm:uzeta}
\end{theorem}

U-labelable posets are a generalization of \emph{R-labelable} posets:
an R-labelable poset is a graded U-labelable poset.  Thus
Theorem~\ref{thm:uzeta} provides fast zeta and M\"obius transforms for
all R-labelable posets, which include supersolvable lattices and
semimodular lattices (see Stanley~\cite[\S 3.14]{stanley2011}).  For a
semimodular lattice an R-labeling is obtained by naming the
join-irreducible elements as $1,2,\ldots,n$ in an order compatible
with the lattice (for example, ordering by rank), and defining
\[
\lambda(s,t) = \min \{i : s \vee i = t\}.
\]

Lower semimodular lattices are U-labelable as well.  More generally,
duality preserves U-labelability: if $\lambda$ is a U-labeling for
poset~$P$, then $\lambda^*(s,t) = -\lambda(t,s)$ is a U-labeling
for~$P^*$.  As an example, the figure below on the left shows an upper
semimodular lattice with a U-labeling.  On the right is its dual with
the derived U-labeling, whereby Algorithm~\ref{alg:addalong} computes
the zeta transform in $9$ additions.  In comparison,
Algorithm~\ref{alg:fastzeta} can use up to $11$ additions (depending
on how $I$ is ordered).

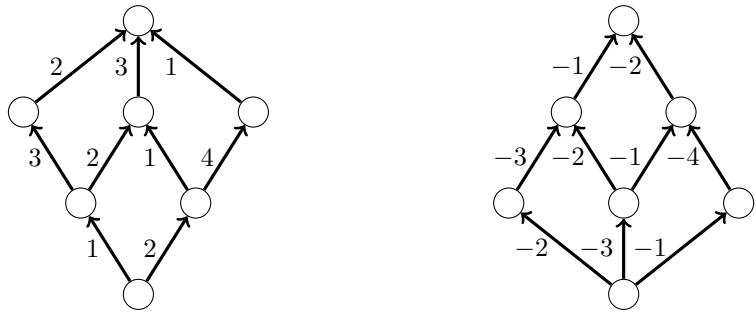
\begin{figure}[H]
  \begin{minipage}{0.49\textwidth}
    \centering
    \begin{tikzpicture}
      \matrix(a)[matrix of math nodes, column sep=0.35cm, row sep=0.8cm]{
        && \node[ele](6){}; \\
        \node[ele](3){}; && \node[ele](5){}; && \node[ele](4){}; \\
        & \node[ele](1){}; && \node[ele](2){}; \\
        && \node[ele](0){}; \\};
      \draw[edgeb] (0)--(1) node[midway,left] {1};
      \draw[edgeb] (2)--(5) node[midway,left] {1};
      \draw[edgeb] (4)--(6) node[midway,left=1mm] {1};
      \draw[edgeb] (0)--(2) node[midway,left] {2};
      \draw[edgeb] (1)--(5) node[midway,left] {2};
      \draw[edgeb] (3)--(6) node[midway,left=1mm] {2};
      \draw[edgeb] (1)--(3) node[midway,left] {3};
      \draw[edgeb] (5)--(6) node[midway,left] {3};
      \draw[edgeb] (2)--(4) node[midway,left] {4};
    \end{tikzpicture}
  \end{minipage}
  \hfill
  \begin{minipage}{0.49\textwidth}
    \centering
    \begin{tikzpicture}
      \matrix(a)[matrix of math nodes, column sep=0.35cm, row sep=0.8cm]{
        && \node[ele](0){}; \\
        & \node[ele](1){}; && \node[ele](2){}; \\
        \node[ele](3){}; && \node[ele](5){}; && \node[ele](4){}; \\
        && \node[ele](6){}; \\};
      \draw[edger] (0)--(1) node[midway,left] {$-1$};
      \draw[edger] (2)--(5) node[midway,left] {$-1$};
      \draw[edger] (4)--(6) node[midway,left=0.5mm] {$-1$};
      \draw[edger] (0)--(2) node[midway,left] {$-2$};
      \draw[edger] (1)--(5) node[midway,left] {$-2$};
      \draw[edger] (3)--(6) node[midway,left=1mm] {$-2$};
      \draw[edger] (1)--(3) node[midway,left] {$-3$};
      \draw[edger] (5)--(6) node[midway,left] {$-3$};
      \draw[edger] (2)--(4) node[midway,left] {$-4$};
    \end{tikzpicture}
  \end{minipage}
\end{figure}

The pentagon lattice, shown below on the left, is not graded but has a
U-labeling, facilitating both transforms in $5$~operations.  The
hexagon lattice (below on the right) cannot be U-labeled.  This would
require rising chains from $p$ to~$r$ and from $q$ to~$s$, implying
that the chain $p \lessdot q \lessdot r \lessdot s$ is rising; but
similarly $p \lessdot t \lessdot u \lessdot s$ would be rising, so
there would be two rising chains from $p$ to~$s$.

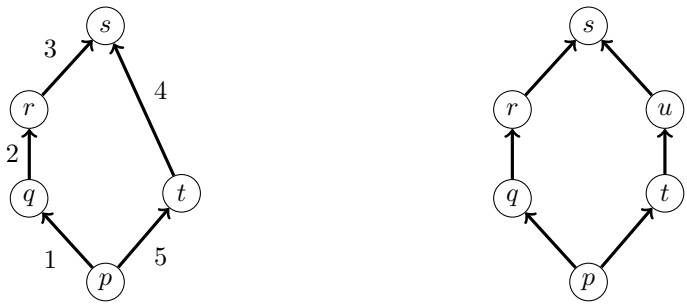
\begin{figure}[H]
  \begin{minipage}{0.49\textwidth}
    \centering
    \begin{tikzpicture}[baseline=(current bounding box.center)]
      \matrix[matrix of nodes, column sep=0.5cm, row sep=0.6cm]{
        & \node[ell](1){$s$}; \\
        \node[ell](b){$r$}; \\
        \node[ell](a){$q$}; && \node[ell](c){$t$}; \\
        & \node[ell](0){$p$}; \\};
      \draw[edgeb] (b)--(1) node[midway,above left] {3};
      \draw[edgeb] (a)--(b) node[midway,left] {2};
      \draw[edgeb] (0)--(a) node[midway,below left] {1};
      \draw[edgeb] (c)--(1) node[midway,above right] {4};
      \draw[edgeb] (0)--(c) node[midway,below right] {5};
    \end{tikzpicture}
  \end{minipage}
  \hfill
  \begin{minipage}{0.49\textwidth}
    \centering
    \begin{tikzpicture}[baseline=(current bounding box.center)]
      \matrix(a)[matrix of nodes, column sep=0.5cm, row sep=0.6cm]{
        & \node[ell](top){$s$}; \\
        \node[ell](4){$r$}; && \node[ell](5){$u$}; \\
        \node[ell](1){$q$}; && \node[ell](2){$t$}; \\
        & \node[ell](bottom){$p$}; \\};
      \draw[edgeb] (bottom)--(1);
      \draw[edgeb] (1)--(4);
      \draw[edgeb] (4)--(top);
      \draw[edgeb] (bottom)--(2);
      \draw[edgeb] (2)--(5);
      \draw[edgeb] (5)--(top);
    \end{tikzpicture}
  \end{minipage}
\end{figure}

In the hexagon it is impossible to compute the zeta transform in $e$
pairwise operations (additions or subtractions); at the minimum, $7 =
e+1$ operations are required.  This can be seen by observing that at
least four operations are required to compute the four outputs at $q,
r, t,$ and $u$, and then enumerating the possibilities of computing the
remaining output $g(s)$.

The converse of Theorem~\ref{thm:uzeta} does not hold for posets in
general.  There are posets that are not U-labelable, but admit the
zeta transform in $e$ operations or less (for example the bipartite
poset mentioned in the introduction, augmented with a hexagon on top).
However, for lattices this seems to be an open question: if a lattice
admits the zeta transform in $e$ operations, is it necessarily
U-labelable?

Another open question concerns whether there are any posets where the
zeta transform requires more than $2e$ operations.  The hardest known
instance, in terms of~$e$, seems to be a lattice constructed from
$k$~parallel chains of $k$~elements, adjoined with a common bottom and
a common top, with a total of $e=k(k+1)$ edges.  If only addition is
available, then Theorem~1 of J\"arvisalo \etal~\cite{jarvisalo2012}
implies that computing the zeta transform for this lattice requires
$2e-O(1)$ additions.

%%%%%%%%%%%%%%%%%%%%%%%%%%%%%%%%%%%%%%%%%%%%%%%%%%%%%%%
\subsection*{Acknowledgements}
The research that led to these results has received funding from the
European Research Council under the European Union's Seventh Framework
Programme (FP/2007-2013) / ERC Grant Agreement 338077 ``Theory and
Practice of Advanced Search and Enumeration.''

We thank Marcus Greferath for the conjecture that became
Theorem~\ref{thm:geom}, and Matti Karppa for helpful comments.

%%%%%%%%%%%%%%%%%%%%%%%%%%%%%%%%%%%%%%%%%%%%%%%%%%%%%%%

\bibliographystyle{plain}
\bibliography{refs}

\end{document}